\documentclass[12pt,a4paper]{article}
\usepackage[utf8]{inputenc}
\usepackage[english]{babel}
\usepackage{amsmath}
\usepackage{amsfonts}
\usepackage{amssymb}
\usepackage{color}

\usepackage{comment}

\usepackage[left=2cm,right=2cm,top=2cm,bottom=2cm]{geometry}

\title{On the Existence of Solution of the Boundary-Domain Integral Equation System derived from the 2D  Dirichlet Problem for the Diffusion Equation with Variable Coefficient}
\author{C.F. Portillo, Z.W. Woldemicheal}
\date{This research was supported with the London Mathematical Society - scheme 5 Collaboration with Developing Countries, grant 51809}

\newenvironment{proof}{\paragraph{Proof:}}{\hfill$\square$}

\newtheorem{theorem}{Theorem}[section]

\newtheorem{lemma}[theorem]{Lemma}
\newtheorem{corollary}[theorem]{Corollary}

\numberwithin{equation}{section}

\allowdisplaybreaks
\begin{document}
\maketitle

\begin{abstract}
A system of boundary-domain integral equations is derived from the bidimensional Dirichlet problem for the diffusion equation with variable coefficient using the novel parametrix from \cite{carlos2} different from the one in \cite{mikhailov1, localised}. Mapping properties of the surface and volume parametrix based potential-type operators are analysed. Invertibility of the single layer potential is also studied in detail in appropriate Sobolev spaces. We show that the system of boundary-domain integral equations derived is equivalent to the Dirichlet problem prescribed and we prove the existence and uniqueness of solution in suitable Sobolev spaces of the system obtained by using arguments of compactness and Fredholm Alternative theory.  
\end{abstract}


\section{Introduction}
Boundary Domain Integral Equation Systems (BDIES) are often derived from a wide class of boundary value problems with variable coefficient in domains with smooth or Lipschitz boundary: cf. \cite{mikhailov1} for a scalar mixed elliptic BVP in bounded domains with smooth boundary; cf. \cite{exterior} for the corresponding problem in unbounded domains with smooth boundary; cf. \cite{carlos3} for the mixed problem in Lipschitz domains. Nevertheless, most of these results only concern three dimensional problems and thus the theoretical work concerning the derivation of BDIES for two dimensional boundary value problems is still being developed.  

Let us note that Boundary Domain Integral Equations (BDIEs) represent a generalisation of the Boundary Integral Equations (BIEs) which are popular due to the reduction of dimension from the domain in which the BVP is defined to its boundary. However, this reduction in dimension only applies for homogeneous BVPs with constant coefficients. As soon as we include variable coefficients or not homogeneous problems the integral equations are defined not only in the boundary but also in the domain of the BVP. Still, one can transform domain integrals into boundary integrals in order to preserve the reduction of dimension using the methods shown in \cite{RIM}. This method is also able to remove various singularities appearing in the domain integrals. 

Also, reformulating the original BVP in the Boundary Domain Integral Equation form can be beneficial, for instance, in inverse problems with variable coefficients, see\cite{chapko}. 

 In order to obtain BIEs, a fundamental solution is required which is not usually explicitly available for problems with variable coefficient and therefore the concept of parametrix is introduced, see \cite{mikhailov1}. A parametrix (see formula \eqref{parametrixdef}) preserves a strong relationship with the corresponding fundamental solution of the analogous BVP with constant coefficient. Using this relationship, it is possible to derive further relations between the surface and volume potential type operators of the variable coefficient case with their counterparts from the constant coefficient case, see, e.g. \cite[Formulae (3.10)-(3.13)]{mikhailov1}, \cite[Formulae (4.6)-(4.11)]{carlos1}. 

A parametrix is not unique. For example, a family of weakly singular parametrices of the form $P^{y}(x,y)$ for the particular operator $\mathcal{A}$ from the diffusion equation in non-homogeneous media
\[\mathcal{A}u(x) := \sum_{i=1}^{3}\dfrac{\partial}{\partial x_{i}}\left(a(x)\dfrac{\partial u(x)}{\partial x_{i}}\right), \,\, P^{y}(x,y)= P(x,y;a(y))=\dfrac{-1}{4\pi a(y)|x-y|}, x, y \in \mathbb{R}^{3}.\]
 has been extensively studied for the 3D case in \cite{mikhailov1,mikhailovlipschitz, exterior}. Note that the superscript in $P^{y}(x,y)$ means that $P^{y}(x,y)$ is a function of the variable coefficient depending on $y$. In this case, the operator $\mathcal{A}$ differentiates with respect to $x$ and the parametrix includes the variable coefficient with respect to $y$. 
 
There is some preliminary results for the analogous operator $\mathcal{A}$ in two dimensions, see \cite{mixed2d, neumann2d, dufera}. In this case, the operator $\mathcal{A}$ and parametrix $P^{y}(x,y)$ reads
\begin{align}
\mathcal{A}u(x) &:= \sum_{i=1}^{2}\dfrac{\partial}{\partial x_{i}}\left(a(x)\dfrac{\partial u(x)}{\partial x_{i}}\right)\label{op1A}, \\ P^{y}(x,y)&= P(x,y;a(y)):=\dfrac{1}{2\pi a(y)}\mathrm{log}|x-y|,\,\, x, y \in \mathbb{R}^{2}\label{pary2d}
\end{align}

In \cite{2dnumerics}, the parametrix \ref{pary2d} has been employed to solve the Dirichlet problem operator \ref{op1A} in 2D. Furthermore, the authors in \cite{2dnumerics} highlight that there is not much research in the literature related to numerical solution of boundary-domain integral equations in 2D obtained by the method presented in this paper. They show that it is possible to obtain linear convergence with respect to the number of quadrature curves, and in some cases, exponential convergence. Moreover, there is analogous research in 3D which succesfully implemented fast algorithms in 3D to obtain the solution of boundary domain integral equations, see \cite{ravnik, numerics, sladek}. Therefore, we believe this method brings new techniques to solve inverse boundary value problems with variable coefficients that can be computationally implemented in an efficient fashion.

 \textit{In this paper,} we explore the family of parametrices for the operator $\mathcal{A}$ of the form $$P^{x}(x,y)= P(x,y;a(x))=\dfrac{1}{2\pi a(x)}\mathrm{log}|x-y|.$$
 which can be useful at the time of studying BDIES derived from a BVP with a system of PDEs with variable coefficient as illustrated in \cite[Section 1]{carlos2}. In particular, the work presented in this paper, will provide a method to obtain an equivalent system of BDIEs even when the single layer potential is not invertible. Although, there is some preliminary work related to BDIEs in two dimensional domains, see \cite{dufera}, this only relates to the family of parametrices $P^{y}(x,y)$ and therefore, the corresponding analysis for the family $P^{x}(x,y)$ in two dimensions is a problem that remains open, and thus is the main purpose of this paper. This study aims to continue the work in \cite{dufera, carlos3} and will motivate the study of BDIEs for the Stokes system in 2D.

In order to study the possible numerical advantages of 
the new family of parametrices of the form $P^{x}(x,y;a(x))$ with respect to the parametrices already studied, it is necessary to prove the unique-solvency of an analogous BDIES derived with this new family of parametrices which has not yet been done for the bidimensional Dirichlet problem for the diffusion equation with variable coefficient.  

The theoretical study of parametrices which include variable coefficient with different variables is helpful at the time of deriving BDIES for boundary value problems for systems of PDEs. For example, the parametrix for the Stokes system involves the variable viscosity coefficient with respect to $x$ and also with respect to $y$, see \cite{carlos1}. 

The main differences between the different families of parametrices are the relations between the parametrix-based potentials with their counterparts for the constant coefficient case.  Notwithstanding, the same mapping properties in Sobolev-Bessel potential spaces still hold allowing us to prove the equivalence between the BDIES and the BVP. 

An analysis of the uniqueness of the BDIES is performed by studying the Fredholm properties of the matrix operator which defines the system.

\section{Preliminaries and the BVP}
Let $\Omega=\Omega^{+}$ be a bounded simply connected domain, $\Omega^{-}:=\mathbb{R}^{2}\smallsetminus\bar{\Omega}^{+}$ the complementary (unbounded) subset of $\Omega$. The boundary $S:=\partial\Omega$ is simply connected, closed and infinitely differentiable, $S\in\mathcal{C}^{\infty}$.  

Let us introduce the following partial differential equation with variable smooth positive coefficient $a(x)\in \mathcal{C}^{\infty}(\overline{\Omega})$:
\begin{equation}\label{ch4operatorA}
\mathcal{A}u(x):=\sum_{i=1}^{2}\dfrac{\partial}{\partial x_{i}}\left(a(x)\dfrac{\partial u(x)}{\partial x_{i}}\right)=f(x),\,\,x\in \Omega,
\end{equation}
 where $u(x)$ is an unknown function and $f$ is a given function on $\Omega$. It is easy to see that if $a\equiv 1$ then, the operator $\mathcal{A}$ becomes $\Delta$, the Laplace operator.
 
 We will use the following function spaces in this paper (see e.g. \cite{mclean, hsiao} for more details). Let $\mathcal{D}'(\Omega)$ be the
Schwartz distribution space; $H^{s}(\Omega)$ and $H^{s}(S)$ with $s\in \mathbb{R}$,
the Bessel potential spaces; the space $H^{s}_{K}(\mathbb{R}^{2})$ consisting of all the distributions of $H^{s}(\mathbb{R}^{2})$ whose support is inside of a compact set $K\subset \mathbb{R}^{2}$; the spaces consisting of distributions in $H^{s}(K)$ for every compact $K\subset \overline{\Omega^{-}},\hspace{0.1em}s\in\mathbb{R}$. We denote $\widetilde{H}^{s}(\Omega)$ the subspace of $H^{s}(\mathbb{R}^{2}),$ 
 $\widetilde{H}^{s}(\Omega)=\lbrace g\in H^{s}(\mathbb{R}^{2}): {\rm supp}(g)\subset \overline{\Omega}\rbrace$.  
 
We will make use of the space, see e.g. \cite{costabel, mikhailov1},
\begin{center}
 $H^{1,0}(\Omega; \mathcal{A}):= \lbrace u \in H^{1}(\Omega): \mathcal{A}u\in L^{2}(\Omega)\rbrace $
 \end{center} 
 which is a Hilbert space with the norm defined by
\begin{center}
$\parallel u \parallel^{2}_{H^{1,0}(\Omega; \mathcal{A})}:=\parallel u \parallel^{2}_{H^{1}(\Omega)}+\parallel \mathcal{A}u  \parallel^{2}_{L^{2}(\Omega)}$.
\end{center}

For a scalar function $w\in H^{s}(\Omega^\pm)$, $s>1/2$, the trace operator $\gamma^{\pm}(\,\cdot \,):=\gamma_{S}^\pm(\,\cdot\,)$, acting on $w$ is well defined and $\gamma^{\pm}w\in H^{s-\frac{1}{2}}(S)$ (see, e.g., \cite{mclean, traces}). 
For $u\in H^{s}(\Omega)$, $s>3/2$, we can define on $S$ the  conormal derivative operator, $T^{\pm}$, in the classical (trace) sense
\begin{equation*}\label{ch4conormal}
T^{\pm}_{x}u := \sum_{i=1}^{2}a(x)\gamma^{\pm}\left( \dfrac{\partial u}{\partial x_{i}}\right)^{\pm}n_{i}^{\pm}(x),
\end{equation*}
where $n^{+}(x)$ is the exterior unit normal vector directed \textit{outwards} the interior domain $\Omega$ at a point $x\in S$. Similarly, $n^{-}(x)$ is the unit normal vector directed \textit{inwards} the interior domain $\Omega$ at a point $x\in S$. 

Furthermore, we will use the notation $T^{\pm}_{x}u$ or $T^{\pm}_{y}u$ to emphasise which respect to which variable we are differentiating. When the variable of differentiation is obvious or is a dummy variable, we will simply use the notation $T^{\pm}u$.

Moreover, for any function $u\in H^{1,0}(\Omega; \mathcal{A})$, the \textit{canonical} conormal derivative $T^{\pm}u\in H^{-\frac{1}{2}}(\Omega)$, is well defined, cf. \cite{costabel,mclean,traces},
\begin{equation}\label{ch4green1}
\langle T^{\pm}u, w\rangle_{S}:= \pm \int_{\Omega^{\pm}}[(\gamma^{-1}\omega)\mathcal{A}u +E(u,\gamma^{-1}w)] dx,\,\, w\in H^{\frac{1}{2}}(S),
\end{equation}
where $\gamma^{-1}: H^{\frac{1}{2}}(S)\longrightarrow H_{K}^{1}(\mathbb{R}^{2})$ is a continuous right inverse to the trace operator whereas the function $E$ is defined as
\begin{equation*}\label{ch4functionalE}
E(u,v)(x):=\sum_{i=1}^{2}a(x)\dfrac{\partial u(x)}{\partial x_{i}}\dfrac{\partial v(x)}{\partial x_{i}},
\end{equation*}
and $\langle\, \cdot \, , \, \cdot \, \rangle_{S}$ represents the $L^{2}-$based dual form on $S$.

We aim to derive boundary-domain integral equation systems for the following \textit{Dirichlet} boundary value problem. Given $f\in L^{2}(\Omega)$ and $\varphi_{0}\in H^{\frac{1}{2}}(\partial\Omega)$, we seek a function $u\in H^{1}(\Omega)$ such that 
\begin{subequations}\label{ch4BVP}
\begin{align}
\mathcal{A}u&=f,\hspace{1em}\text{in}\hspace{1em}\Omega\label{ch4BVP1};\\
\gamma^{+}u &= \varphi_{0},\hspace{1em}\text{on}\hspace{1em} \partial\Omega\label{ch4BVP2}
\end{align}
where equation \eqref{ch4BVP1} is understood in the weak sense, the Dirichlet condition \eqref{ch4BVP2} is understood in the trace sense.
\end{subequations}

By Lemma 3.4 of \cite{costabel} (cf. also Theorem 3.9 in \cite{traces}), the first Green identity holds for any $u\in H^{1,0}(\Omega; \mathcal{A})$ and $v\in H^{1}(\Omega)$,
\begin{equation}\label{ch4green1.1}
\langle T^{\pm}u, \gamma^{+}v\rangle_{S}:= \pm \int_{\Omega}[v\mathcal{A}u +E(u,v)] dx.
\end{equation}
The following assertion is well known and can be proved, e.g., using the Lax-Milgram lemma as in \cite[Chapter 4]{steinbach}. 
\begin{theorem}\label{ch4Thomsol}
The boundary value problem \eqref{ch4BVP} has one and only one  solution.
\end{theorem}
 
\section{Parametrices and remainders}
 We define a parametrix (Levi function) $P(x,y)$ for a differential operator $\mathcal{A}_{x}$ differentiating with respect to $x$ as a function on two variables that satisfies
 \begin{equation}\label{parametrixdef}
 \mathcal{A}_{x}P(x,y) = \delta(x-y)+R(x,y).
 \end{equation}
 where $\delta (.)$ is a Dirac-delta distribution, while $R(x,y)$ is a remainder possessing at most a weak (integrable) singularity at $ x=y $.
 
For a given operator $\mathcal{A}$, the parametrix is not unique. For example, the parametrix 
\begin{equation*}\label{ch4P2002}
P^y(x,y)=\dfrac{1}{a(y)} P_\Delta(x-y),\hspace{1em}x,y \in \mathbb{R}^{2},
\end{equation*} 
was employed in \cite{localised, mikhailov1}, for the operator $\mathcal{A}$ defined in \eqref{ch4operatorA},  where
\begin{equation*}\label{ch4fundsol}
P_{\Delta}(x-y) = \dfrac{1}{2\pi}\mathrm{log}|x-y|
\end{equation*}
is the fundamental solution of the Laplace operator.
The remainder corresponding to the parametrix $P^{y}$ is 
\begin{equation*}
\label{ch43.4} R^y(x,y)
=\sum\limits_{i=1}^{2}\frac{1}{a(y)}\, \frac{\partial a(x)}{\partial x_i} \frac{\partial }{\partial x_i}P_\Delta(x-y)
\,,\;\;\;x,y\in {\mathbb R}^2.
\end{equation*}
{\em In this paper}, for the same operator $\mathcal{A}$ defined in \eqref{ch4operatorA}, we will use another parametrix,  
\begin{align}\label{ch4Px}
P(x,y):=P^x(x,y)=\dfrac{1}{a(x)} P_\Delta(x-y),\hspace{1em}x,y \in \mathbb{R}^{2},
\end{align}
which leads to the corresponding remainder 
\begin{align*}\label{ch4remainder}
R(x,y) =R^x(x,y) &= 
-\sum\limits_{i=1}^{2}\dfrac{\partial}{\partial x_{i}}
\left(\frac{1}{a(x)}\dfrac{\partial a(x)}{\partial x_{i}}P_{\Delta}(x,y)\right)\\
& =
-\sum\limits_{i=1}^{2}\dfrac{\partial}{\partial x_{i}}
\left(\dfrac{\partial \ln a(x)}{\partial x_{i}}P_{\Delta}(x,y)\right),\hspace{0.5em}x,y \in \mathbb{R}^{2}.
\end{align*}

Note that the both remainders $R_x$ and $R_y$ are weakly singular, i.e., \[
R^x(x,y),\,R^y(x,y)\in \mathcal{O}(\vert x-y\vert^{-2}).\]
 This is due to the smoothness of the variable coefficient $a$.

\section{Volume and surface potentials}

The parametrix-based logarithmic and remainder potential operators are respectively defined, similar to \cite{{mikhailov1},{carlos2}} in the 3D case for $y\in\mathbb R^2$, as
\begin{align*}
\mathcal{P}\rho(y)&:=\displaystyle\int_{\Omega} P(x,y)\rho(x)\hspace{0.25em}dx\\
\mathcal{R}\rho(y)&:=\displaystyle\int_{\Omega} R(x,y)\rho(x)\hspace{0.25em}dx.
\end{align*}

 The parametrix-based single layer and double layer  surface potentials are defined for $y\in\mathbb R^2:y\notin S $, as 
\begin{equation*}\label{ch4SL}
V\rho(y):=-\int_{S} P(x,y)\rho(x)\hspace{0.25em}dS(x),
\end{equation*}
\begin{equation*}\label{ch4DL}
W\rho(y):=-\int_{S} T_{x}^{+}P(x,y)\rho(x)\hspace{0.25em}dS(x).
\end{equation*}

We also define the following pseudo-differential operators associated with direct values of the single and  double layer potentials and with their conormal derivatives, for $y\in S$,
\begin{align}
\mathcal{V}\rho(y)&:=-\int_{S} P(x,y)\rho(x)\hspace{0.25em}dS(x),\nonumber \\
\mathcal{W}\rho(y)&:=-\int_{S} T_{x}P(x,y)\rho(x)\hspace{0.25em}dS(x),\nonumber\\
\mathcal{W'}\rho(y)&:=-\int_{S} T_{y}P(x,y)\rho(x)\hspace{0.25em}dS(x),\nonumber\\
\mathcal{L}^{\pm}\rho(y)&:=T_{y}^{\pm}{W}\rho(y)\nonumber.
\end{align}

The operators $\mathcal P, \mathcal R, V, W, \mathcal{V}, \mathcal{W}, \mathcal{W'}$ and $\mathcal{L}$ can be expressed in terms the volume and surface potentials and operators associated with the Laplace operator, as follows
\begin{align}
\mathcal{P}\rho&=\mathcal{P}_{\Delta}\left(\dfrac{\rho}{a}\right),\label{ch4relP}\\
\mathcal{R}\rho&=\nabla\cdot\left[\mathcal{P}_{\Delta}(\rho\,\nabla \ln a)\right]-\mathcal{P}_{\Delta}(\rho\,\Delta \ln a),\label{ch4relR}\\
V\rho &= V_{\Delta}\left(\dfrac{\rho}{a}\right),\label{ch4relSL}\\
\mathcal{V}\rho &= \mathcal{V}_{\Delta} \left( \dfrac{\rho}{a}\right),\label{ch4relDVSL}\\
W\rho &= W_{\Delta}\rho -V_{\Delta}\left(\rho\frac{\partial \ln a}{\partial n}\right),\label{ch4relDL}\\
\mathcal{W}\rho &= \mathcal{W}_{\Delta}\rho -\mathcal{V}_{\Delta}\left(\rho\frac{\partial \ln a}{\partial n}\right),\label{ch4relDVDL}\\
\mathcal{W}'\rho &= a \mathcal{W'}_{\Delta}\left(\dfrac{\rho}{a}\right),\label{ch4relTSL} \\
\mathcal{L}^{\pm}\rho &= \widehat{\mathcal{L}}\rho - aT^{\pm}_\Delta V_{\Delta}\left(\rho\frac{\partial \ln a}{\partial n}\right),
\label{ch4relTDL}\\
\widehat{\mathcal{L}}\rho &:= a\mathcal{L}_{\Delta}\rho.\label{ch4hatL}
\end{align}

The symbols with the subscript $\Delta$ denote the analogous operator for the constant coefficient case, $a\equiv 1$. Furthermore, by the Liapunov-Tauber theorem, $\mathcal{L}_{\Delta}^{+}\rho = \mathcal{L}_{\Delta}^{-}\rho = \mathcal{L}_{\Delta}\rho$.

Using relations \eqref{ch4relP}-\eqref{ch4hatL} it is now rather simple to obtain, similar to \cite{mikhailov1}, the mapping properties, jump relations and invertibility results for the parametrix-based  surface and volume potentials, provided in theorems/corollary \ref{ch4thmUR}-\ref{ch4thcompVW}, from the well-known properties of their constant-coefficient counterparts (associated with the Laplace equation).
 
 \begin{theorem}\label{ch4thmUR} Let $s\in \mathbb{R}$. Then, the following operators are continuous,
 \begin{equation*}
 \mathcal{P}:\widetilde{H}^{s}(\Omega) \longrightarrow H^{s+2}(\Omega),\hspace{0.5em} s\in \mathbb{R},\label{ch4mpvp1}\\
 \end{equation*}
 \begin{equation*}
 \mathcal{P}: H^{s}(\Omega) \longrightarrow H^{s+2}(\Omega),\hspace{0.5em} s>-\dfrac{1}{2},\label{ch4mpvp2}\\
 \end{equation*}
 \begin{equation*}
 \mathcal{R}:\widetilde{H}^{s}(\Omega) \longrightarrow H^{s+1}(\Omega),\hspace{0.5em} s\in \mathbb{R},\label{ch4mpvp3}\\
 \end{equation*}
 \begin{equation*}
 \mathcal{R}: H^{s}(\Omega) \longrightarrow H^{s+1}(\Omega),\hspace{0.5em} s>-\dfrac{1}{2}\,.\label{ch4mpvp4}
 \end{equation*}
 \end{theorem}
 \begin{corollary}\label{ch4corcompact} The following operators are compact for any $s>\frac{1}{2}$,
 \begin{align*}
 \mathcal{R}&: H^{s}(\Omega) \longrightarrow H^{s}(\Omega),\\
 \gamma^{+}\mathcal{R}&: H^{s}(\Omega) \longrightarrow H^{s-\frac{1}{2}}(S),\\
 T^{+}\mathcal{R}&: H^{s}(\Omega) \longrightarrow H^{s-\frac{3}{2}}(S).
 \end{align*}
 \end{corollary}   

\begin{theorem}\label{ch4thmappingVW} Let $s\in \mathbb{R}$. Then, the following operators are continuous:
\begin{align*}
V: H^{s}(S) \longrightarrow H^{s+\frac{3}{2}}(\Omega),\\
W: H^{s}(S) \longrightarrow H^{s+\frac{1}{2}}(\Omega).
\end{align*}
\end{theorem}

\begin{theorem}\label{ch4thmappingDVVW} Let $s\in \mathbb{R}$. Then, the following operators are continuous:
\begin{align*}
\mathcal{V}&: H^{s}(S) \longrightarrow H^{s+1}(S),\\
\mathcal{W}&: H^{s}(S) \longrightarrow H^{s+1}(S),\\
\mathcal{W'}&: H^{s}(S) \longrightarrow H^{s+1}(S),\\
\mathcal{L}^\pm&: H^{s}(S) \longrightarrow H^{s-1}(S).
\end{align*}
\end{theorem}

\begin{theorem}\label{ch4thjumps} Let  $\rho\in H^{-\frac{1}{2}}(S)$, $\tau\in H^{\frac{1}{2}}(S)$. Then the following operators jump relations hold:
\begin{align*}
\gamma^{\pm}V\rho&=\mathcal{V}\rho,\\
\gamma^{\pm}W\tau&=\mp\dfrac{1}{2}\tau+\mathcal{W}\tau,\\
T^{\pm}V\rho&=\pm\dfrac{1}{2}\rho+\mathcal{W'}\rho.
\end{align*}
\end{theorem}

\begin{theorem}\label{ch4thcompVW} Let $s\in \mathbb{R}.$ The following operators 
\begin{align*}
\mathcal{V}: H^{s}(S) \longrightarrow H^{s}(S),\\
\mathcal{W}: H^{s}(S) \longrightarrow H^{s}(S),\\
\mathcal{W}':H^{s}(S) \longrightarrow H^{s}(S).
\end{align*}
are compact.
\end{theorem}
\section{Invertibility of the single layer potential operator}
It is well-known that for some 2D domains the kernel of the operator $\mathcal{V}_{\Delta}$ is non-zero, which by relation \eqref{ch4relSL} also implies that the kernel of the operator $\mathcal{V}$ is also non-zero for the same domain (see e.g. \cite[Remark 1.42(ii)]{costanda}, \cite[proof of Theorem 6.22]{steinbach}, \cite{dufera}).

Since the boundary integral operator $\mathcal{V}$ has the non-trivial kernel on some two dimensional domains, we have to consider the boundary integral operator in suitable spaces. Thus in order to have invertibility for the single layer potential operator in two dimension, we define the following subspace of the space $H^{-\frac{1}{2}}(S)$ (see, e.g.,\cite[Eq. (6.30)]{steinbach},
\begin{equation*}
H_{*}^{-\frac{1}{2}}(S):=\left\lbrace \phi \in H^{-\frac{1}{2}}(S):\langle \phi,1\rangle_{S}=0\right\rbrace ,
\end{equation*}
where the norm in $H_{*}^{-\frac{1}{2}}(S)$ is the induced by the norm in $H^{-\frac{1}{2}}(S)$.
\begin{theorem}\label{ch4thinv2DV1}
Let $\psi \in H_{*}^{-\frac{1}{2}}(\partial\Omega)$ satisfies $\mathcal{V}\psi =0$ on $\partial\Omega$, then $\psi =0$.
\end{theorem}
\begin{proof}
Relation \eqref{ch4relSL} gives $\mathcal{V}g = \mathcal{V}_{\Delta}g^{*}$, where $g = g^{*}/a$. The invertibility of $\mathcal{V}$ then follows from the invertibility of $\mathcal{V}_{\Delta}$, see references \cite[Theorem 2.4]{costste}, \cite[Theorem 3.5]{miksolandreg} and \cite[Theorem 4]{dufera}.
\end{proof}
\begin{theorem}\label{ch4thinv2DV2}
Let $\Omega \subset \mathbb{R}^{2}$ have the diameter $\mbox{diam}(\Omega) <1$. Then the single layer potential $\mathcal{V}:H^{-\frac{1}{2}}(\partial\Omega)\rightarrow H^{\frac{1}{2}}(\partial\Omega)$ is invertible.
\end{theorem}
\begin{proof}
The proof is similar to the ones in \cite{dufera} but for the different parametrix \eqref{ch4Px} we have the relation \eqref{ch4relSL} and the invertibility of the operator $\mathcal{V}:H^{-\frac{1}{2}}(\partial\Omega) \rightarrow H^{\frac{1}{2}}(\partial\Omega)$ also follows.  
\end{proof}
\section{Third Green identities and integral relations}

In this section we provide the results similar to the ones in \cite{mikhailov1} but for our, different, parametrix \eqref{ch4Px}.  
  
Let  $u,v\in H^{1,0}(\Omega;\mathcal{A})$.   Subtracting from the first Green identity \eqref{ch4green1.1} its counterpart with the swapped $u$ and $v$, we arrive at the second Green identity, see e.g. \cite{mclean},
\begin{equation}\label{ch4green2}
\displaystyle\int_{\Omega}\left[u\,\mathcal{A}v - v\,\mathcal{A}u\right]dx= \int_{S}\left[u\, T^{+}v \,-\,v\, T^{+}u\,\right]dS(x). 
\end{equation}
Taking now $v(x):=P(x,y)$, we obtain from \eqref{ch4green2} by the standard limiting procedures (cf. \cite{miranda}) the third Green identity  for any function $u\in H^{1,0}(\Omega;\mathcal{A})$:
\begin{equation}\label{ch4green3}
u+\mathcal{R}u-VT^{+}u+W\gamma^{+}u=\mathcal{P}\mathcal{A}u,\hspace{1em}\text{in}\hspace{0.2em}\Omega.
\end{equation}

If $u\in H^{1,0}(\Omega; \mathcal{A})$ is a solution of the partial differential equation \eqref{ch4BVP1}, then, from \eqref{ch4green3} we obtain:

\begin{equation}\label{ch43GV}
u+\mathcal{R}u-VT^{+}u+W\gamma^{+}u=\mathcal{P}f,\hspace{0.5em}in\hspace{0.2em}\Omega;
\end{equation}
\begin{equation}\label{ch43GG}
\dfrac{1}{2}\gamma^{+}u+\gamma^{+}\mathcal{R}u-\mathcal{V}T^{+}u+\mathcal{W}\gamma^{+}u=\gamma^{+}\mathcal{P}f,\hspace{0.5em}on\hspace{0.2em}S.
\end{equation}

For some distributions $f$, $\Psi$ and $\Phi$, we consider a more general, indirect integral relation associated with the third Green identity \eqref{ch43GV}:

\begin{equation}\label{ch4G3ind}
u+\mathcal{R}u-V\Psi+W\Phi=\mathcal{P}f,\hspace{0.5em}{\rm in\ }\Omega.
\end{equation}

\begin{lemma}\label{ch4lema1}Let $u\in H^{1}(\Omega)$, $f\in L_{2}(\Omega)$, $\Psi\in H^{-\frac{1}{2}}(S)$ and $\Phi\in H^{\frac{1}{2}}(S)$ satisfying the relation \eqref{ch4G3ind}. Then $u$ belongs to $H^{1,0}(\Omega, \mathcal{A})$; solves the equation $\mathcal{A}u=f$ in $\Omega$, and the following identity is satisfied,
\begin{equation}\label{ch4lema1.0}
V(\Psi- T^{+}u) - W(\Phi- \gamma^{+}u) = 0\hspace{0.5em}\text{in}\hspace{0.5em}\Omega.
\end{equation}
\end{lemma}
\begin{proof}
The proof follows word for word the corresponding proof in 3D case in \cite{carlos2}.
\end{proof}

\begin{lemma}\label{ch4lemma2}
Let either $\Psi^{*} \in H^{-\frac{1}{2}}(S)$ and $\mbox{diam}(\Omega)<1,$ or $\Psi^{*} \in H_{*}^{-\frac{1}{2}}(S)$. If
\begin{equation}\label{ch4lema2i}
V\Psi^{*}(y) = 0,\hspace{2em}y\in\Omega
\end{equation}
then $\Psi^{*}(y) = 0$.
\end{lemma}
\begin{proof} Taking the trace of \eqref{ch4lema2i} gives:
\begin{center}
$\mathcal{V}\Psi^{*}(y) = \mathcal{V}_{\triangle}\left(\dfrac{\Psi^{*}}{a}\right)(y) = 0, \hspace{2em}y\in\Omega$,
\end{center}
If $\Psi^{*} \in H^{-\frac{1}{2}}(S)$ and $\mbox{diam}(\Omega)<1,$ then the result follows from invertibility of the single layer potential given by Theorem \ref{ch4thinv2DV2}. On the other hand, if $\Psi^{*} \in H_{*}^{-\frac{1}{2}}(S),$ then the result is implied by Theorem \ref{ch4thinv2DV1}.   
\end{proof}

\section{BDIE system for the Dirichlet problem}
We aim to obtain a segregated boundary-domain integral equation system for Dirichlet BVP \eqref{ch4BVP}. 
Let us denote the unknown conormal derivative as $\psi:=T^{+}u\in H^{-\frac{1}{2}}(S)$ and we will further consider $\psi$ as formally independent of $u$ in $\Omega$.

To obtain one of the possible boundary-domain integral equation systems we employ identity \eqref{ch43GV} in the domain $\Omega$, and identity \eqref{ch43GG} on $S$, substituting there the Dirichlet condition and $T^{+}u = \psi$ and further considering the unknown function $\psi$ as formally independent (segregated) of $u$ in $\Omega$. Consequently, we obtain the following system (A1) of two equations for two unknown functions,
\begin{subequations}
\begin{align}
u+\mathcal{R}u-V\psi&=F_{0}\hspace{2em}in\hspace{0.5em}\Omega,\label{ch4SM12v}\\
\gamma^{+}\mathcal{R}u-\mathcal{V}\psi&=\gamma^{+}F_{0}-\varphi_{0}\label{ch4SM12g}\hspace{2em}on\hspace{0.5em}S,
\end{align}
\end{subequations}
where
\begin{equation}\label{ch4F0term}
F_{0}=\mathcal{P}f-W\varphi_{0}.
\end{equation}

 We remark that $F_{0}$ belongs to the space $H^{1}(\Omega)$ in virtue of the mapping properties of the surface and volume potentials, see Theorems \ref{ch4thmUR} and \ref{ch4thmappingVW}.

The system (A1), given by \eqref{ch4SM12v}-\eqref{ch4SM12g} can be written in matrix notation as 
\begin{equation*}
\mathcal{A}^{1}\mathcal{U}=\mathcal{F}^{1},
\end{equation*}
where $\mathcal{U}$ represents the vector containing the unknowns of the system,
\begin{equation*}
\mathcal{U}=(u,\psi)^{\top}\in H^{1}(\Omega)\times H^{-\frac{1}{2}}(S),
\end{equation*}
the right hand side vector is \[\mathcal{F}^{1}:= [ F_{0}, \gamma^{+}F_{0} - \varphi_{0} ]^{\top}\in H^{1}(\Omega)\times H^{\frac{1}{2}}(S),\]
and the matrix operator $\mathcal{A}^{1}$ is defined by:
\begin{equation*}
   \mathcal{A}^{1}=
  \left[ {\begin{array}{ccc}
   I+\mathcal{R} & -V  \\
   \gamma^{+}\mathcal{R} & -\mathcal{V} 
  \end{array} } \right].
\end{equation*}

We note that the mapping properties of the operators involved in the matrix imply the continuity of the operator $
    \mathcal{A}^{1}.$

Let us prove that BVP\eqref{ch4BVP} in $\Omega$ is equivalent to the system of BDIEs \eqref{ch4SM12v}-\eqref{ch4SM12g}.
\begin{theorem}\label{ch4EqTh}
Let $f\in L_{2}(\Omega)$ and $\varphi_{0}\in H^{\frac{1}{2}}(S)$. 
\begin{enumerate}
\item[i)] If some $u\in H^{1}(\Omega)$ solves the BVP \eqref{ch4BVP}, then the pair $(u, \psi)^{\top}\in H^{1}(\Omega)\times H^{-\frac{1}{2}}(S)$ where
\begin{equation}\label{ch4eqcond}
\psi=T^{+}u,\hspace{2em}on\hspace{0.5em}S,
\end{equation}
solves the BDIE system (A1). 

\item[ii)] If a couple $(u, \psi)^{\top}\in H^{1}(\Omega)\times H^{-\frac{1}{2}}(S)$ solves the BDIE system (A1), and $\mbox{diam}(\Omega)<1,$ then $u$ solves the BVP and the functions $\psi$ satisfy \eqref{ch4eqcond}.

\item[iii)] The system (A1) is uniquely solvable. 
\end{enumerate}
\end{theorem}

\begin{proof}
First, let us prove item $i)$. Let $u\in H^{1}(\Omega)$ be a solution of the boundary value problem \eqref{ch4BVP} and let $\psi$ be defined by \eqref{ch4eqcond} evidently implies $\psi\in H^{-\frac{1}{2}}(S).$  Then, it immediately follows from the third Green identities \eqref{ch43GV} and \eqref{ch43GG} that the couple $(u, \psi)$ solves BDIE system (A1).

Let us prove now item $ii)$. Let the couple $(u, \psi)^{\top}\in H^{1}(\Omega)\times H^{-\frac{1}{2}}(S)$ solve the BDIE system (A1). Taking the trace of the equation \eqref{ch4SM12v} and substract it from the equation \eqref{ch4SM12g}, we obtain
\begin{equation}\label{ch4M12a1}
\gamma^{+}u=\varphi_{0}, \hspace{1em} \text{on}\hspace{0.5em}S.
\end{equation}
Thus, the Dirichlet boundary condition in \eqref{ch4BVP2} is satisfied. 

We proceed using the Lemma \ref{ch4lema1} in the first equation of the system (A1), \eqref{ch4SM12v}, which implies that $u$ is a solution of the equation \eqref{ch4BVP1} and also the following equality:
\begin{equation*}\label{ch4M12a2}
V(\psi - T^{+}u) - W(\varphi_{0} -\gamma^{+}u) = 0 \text{ in } \Omega.
\end{equation*}
By virtue of \eqref{ch4M12a1}, the second term of the previous equation vanishes. Hence,
\begin{equation*}\label{ch4M12a3}
V(\psi - T^{+}u)= 0, \quad \text{ in } \Omega.
\end{equation*}
 Lemma \ref{ch4lemma2} then implies 
\begin{equation}\label{ch4M12a4}
\psi = T^{+}u,\quad\text{ on } S.
\end{equation}

Item $iii)$ immediately follows from the uniqueness of the solution of the Dirichlet boundary value problem Theorem \ref{ch4Thomsol}.
\end{proof}

\begin{lemma}\label{ch4remark}$(F_{0}, \gamma^{+}F_{0}-\varphi_{0})=0$ if and only if $(f, \varphi_{0})=0$
\end{lemma}
\begin{proof}
Indeed the latter equality evidently implies the former, i.e., if $(f, \varphi_{0})=0$ then $(F_{0}, \gamma^{+}F_{0}-\varphi_{0})=0$. Conversely, supposing that $(F_{0}, \gamma^{+}F_{0}-\varphi_{0})=0$, then taking into account equation \eqref{ch4F0term} and applying Lemma \ref{ch4lema1} with $F_{0}=0$ as $u$, we deduce that $f=0$ and $W\varphi_{0}=0$ in $\Omega$. Now, the second equality, $\gamma^{+}F_{0}-\varphi_{0}=0$, implies that $\varphi_{0}=0$ on $S.$ 
\end{proof}

\begin{theorem}
If $\mbox{diam}(\Omega)<1,$ then the operators are invertible, 
\begin{equation}
\mathcal{A}^{1}: H^{1}(\Omega)\times H^{-\frac{1}{2}}(S)\to H^{1}(\Omega)\times H^{\frac{1}{2}}(S) \label{ch4M12a5}
\end{equation}
\begin{equation}
\hspace{1.5cm}\mathcal{A}^{1}: H^{1,0}(\Omega ;A)\times H^{-\frac{1}{2}}(S)\to H^{1,0}(\Omega ;A)\times H^{\frac{1}{2}}(S)\label{ch4M12a6}
\end{equation}
\end{theorem}
\begin{proof}
To prove the invertibility of operator \eqref{ch4M12a5}, let $\mathcal{A}_{0}^{1}$ be the matrix operator defined by
\begin{equation*}\label{ch4M012}
  \mathcal{A}_{0}^{1}:=
  \left[ {\begin{array}{ccc}
   I & -V  \\
   0 & -\mathcal{V} \\
  \end{array} } \right].
\end{equation*}
As a result of compactness properties of the operators $\mathcal{R}$ and $\gamma^{+}\mathcal{R}$ (cf. Corollary \ref{ch4corcompact}), the operator $\mathcal{A}_{0}^{1}$ is a compact perturbation of operator $\mathcal{A}^{1}$. The operator  $\mathcal{A}_{0}^{1}$ is an upper triangular matrix operator and invertibility of its diagonal operators $I:H^{1}(\Omega)\longrightarrow H^{1}(\Omega)$ and $\mathcal{V}:H^{-\frac{1}{2}}(\partial\Omega)\longrightarrow H^{\frac{1}{2}}(\partial\Omega)$ (cf. Theorem \ref{ch4thinv2DV2}). This implies that 
\begin{center}
$\mathcal{A}_{0}^{1}: H^{1}(\Omega)\times H^{-\frac{1}{2}}(S)\longrightarrow H^{1}(\Omega)\times H^{\frac{1}{2}}(S)$
\end{center}
is an invertible operator. Thus $\mathcal{A}^{1}$ is a Fredholm operator with zero index. Hence the Fredholm property and the injectivity of the operator $\mathcal{A}^{1}$, provided by item $iii)$ of Lemma \ref{ch4remark}, imply the invertibility of operator $\mathcal{A}^{1}$.

To prove invertibility of operator \eqref{ch4M12a6}, we remark that for any $\mathcal{F}^{1}\in H^{1,0}(\Omega ;A)\times H^{\frac{1}{2}}(S)$ a solution of the equation $\mathcal{A}^{1}\mathcal{U}=\mathcal{F}^{1}$ can be written as $\mathcal{U}=\left(\mathcal{A}^{1}\right)^{-1}\mathcal{F}^{1}$, where $
\left(\mathcal{A}^{1}\right)^{-1} : H^{1}(\Omega)\times H^{\frac{1}{2}}(\partial\Omega)\to H^{1}(\Omega)\times H^{-\frac{1}{2}}(\partial\Omega)$
is the continuous inverse to operator \eqref{ch4M12a5}. But due to Lemma \ref{ch4lema1} the first equation of system (A1) implies that $\mathcal{U}=\left(\mathcal{A}^{1}\right)^{-1}\mathcal{F}^{1}\in H^{1,0}(\Omega ;A)\times H^{-\frac{1}{2}}(S)$ and moreover, the operator
$\left(\mathcal{A}^{1}\right)^{-1} : H^{1,0}(\Omega ;A)\times H^{\frac{1}{2}}(S)\to H^{1,0}(\Omega ;A)\times H^{-\frac{1}{2}}(S)$
is continuous, which implies invertibility of operator \eqref{ch4M12a6}.
\end{proof}

\section{Conclusions}
In this paper, we have considered a new parametrix for the Dirichlet problem with variable coefficient in two-dimensional domain, where the right hand side function is from $L_{2}(\Omega)$ and the Dirichlet data from the space $H^{\frac{1}{2}}(S).$  A BDIEs for the original BVP has been obtained. Results of equivalence between the BDIES and the BVP has been shown along with the invertibility of the matrix operator defining the BDIES. 

Now, we have obtained an analogous system to the BDIES (A1) of \cite{mikhailov1,carlos2} with a new family of parametrices which is uniquely solvable. Hence, further investigation about the numerical advantages of using one family of parametrices over another will follow.

Analogous results could be obtain for exterior domains following a similar approach as in \cite{exterior}. 

Further generalised results for Lipschitz domains and non-smooth coefficient can also be obtain by using the generalised canonical conormal derivative operator defined in \cite{traces, mikhailovlipschitz}.

\paragraph{Z. W. Woldemicheal$^1$, C.F. Portillo$^2$,\vspace{5pt}\\}
\begin{tabular}{ll}
$^1$ & {Department of Mathematics}  \\
&{Addis Ababa University} \\  
&{Ethiopia}\\
$^2$ & {School of Engineering, Computing and Mathematics}  \\
&{Oxford Brookes University} \\ 
&{UK } 
\end{tabular}

\end{document}